\documentclass[a4paper,12pt]{amsart}
\usepackage[utf8x]{inputenc}
\usepackage{mathrsfs}
\usepackage{amsfonts,amssymb,amscd,amsthm,amsmath,graphicx}
\usepackage{verbatim}
\usepackage{longtable}
\usepackage[all]{xy}
\usepackage{verbatim}

\theoremstyle{plain}
\newtheorem{theorem}{Theorem}[section]
\newtheorem{corollary}{Corollary}[section]
\newtheorem{lemma}{Lemma}[section]
\newtheorem{proposition}{Proposition}[section]

\newtheorem{question}{Question}[section]

\newtheorem{example}{Example}[section]

\def\E{\operatorname{E}}
\def\F{\operatorname{F}}
\def\G{\operatorname{G}}
\def\GL{\operatorname{GL}}
\def\GSp{\operatorname{GSp}}
\def\HSpin{\operatorname{HSpin}}

\def\Pin{\operatorname{Pin}}
\def\PO{\operatorname{PO}}
\def\PSL{\operatorname{PSL}}
\def\PSO{\operatorname{PSO}}
\def\PSp{\operatorname{PSp}}
\def\PSU{\operatorname{PSU}}
\def\PU{\operatorname{PU}}
\def\O{\operatorname{O}}
\def\SL{\operatorname{SL}}
\def\SO{\operatorname{SO}}
\def\Sp{\operatorname{Sp}}
\def\SU{\operatorname{SU}}
\def\U{\operatorname{U}}

\def\Ad{\operatorname{Ad}}
\def\ad{\operatorname{ad}}

\def\Clif{\operatorname{Clif}}

\def\diag{\operatorname{diag}}
\def\der{\operatorname{der}}

\def\Gal{\operatorname{Gal}}

\def\id{\operatorname{id}}

\def\Im{\operatorname{Im}}

\def\Map{\operatorname{Map}}

\def\pr{\operatorname{pr}}

\def\Spin{\operatorname{Spin}}

\def\tr{\operatorname{tr}}

\begin{document}

\title{Acceptable compact Lie groups}
\thanks{}

\author{Jun Yu}
\address{Jun Yu, Beijing International Center for Mathematical Research, Peking University, No. 5 Yiheyuan Road, Beijing 100871,
China.}
\email{junyu@math.pku.edu.cn}
\keywords{Element-conjugacy, global conjugacy, acceptable group, pseudocharacter.}
\subjclass[2010]{22E46.}

\begin{abstract}
In this paper we show that for a connected compact Lie group to be acceptable it is necessary and sufficient that its derived
subgroup is isomorphic to a direct product of the groups $\SU(n)$, $\Sp(n)$, $\SO(2n+1)$, $\G_2$, $\SO(4)$. We show that there
are invariant functions on $\SO_{4}(\mathbb{C})^{2}$ which are not generated by 1-argument invariants, though the group
$\SO_{4}(\mathbb{C})$ is acceptable.
\end{abstract}

\maketitle

\tableofcontents

\section{Introduction}

Let $G$ be a Lie group and $\Gamma$ be a group. Two homomorphisms $\phi_{1},\phi_{2}:\Gamma\rightarrow G$ are said to be
element-conjugate if $\phi_{1}(x)\sim\phi_{2}(x)\ (\forall x\in\Gamma)$. They are said to be globally conjugate if there
exists $g\in G$ such that $\phi_{2}(x)=g\phi_{1}(x)g^{-1}$ ($\forall x\in\Gamma$). We call a Lie group $G$ acceptable if
element-conjugacy implies global conjugacy for every finite group $\Gamma$ and every pair of homomorphisms $\phi_{1},
\phi_{2}: \Gamma\rightarrow G$; otherwise we call $G$ unacceptable. We call a Lie group $G$ strongly acceptable if
element-conjugacy implies global conjugacy for every group $\Gamma$ and every pair of homomorphisms $\Gamma\rightarrow G$.

The notions of ``acceptable groups" and ``unacceptable groups" are defined by Michael Larsen. In \cite{Larsen1} and
\cite{Larsen2}, Larsen showed that many connected compact Lie groups are acceptable (resp. unacceptable). Particularly,
he classified compact connected and simply-connected Lie groups which are acceptable. In \cite{Fang-Han-Sun} a notion
weaker than``strongly acceptable" was studied where they considered only homomorphisms from connected compact Lie groups
to a given Lie group. In this paper, we take a further study by classifying acceptable connected compact Lie groups,
which amounts to the following Theorem \ref{T1}. We also show in Section \ref{S:nonC} that a few non-connected compact
Lie groups are unacceptable.

\begin{theorem}\label{T1}
Let $G$ be a connected compact Lie group. For $G$ to be acceptable it is necessary and sufficient that its derived subgroup
$G_{\der}=[G,G]$ is isomorphic to a direct product of the following groups:
\[\SU(n),\ \Sp(n),\ \SO(2n+1),\ \G_2,\ \SO(4).\] On the other hand, if $G_{\der}$ is isomorphic to a direct product of the
above groups, then $G$ is strongly acceptable.
\end{theorem}

An immediate consequence of Theorem \ref{T1} is: a connected compact Lie group is acceptable if and only if it is strongly
acceptable. We don't know if this statement holds for non-connected compact Lie groups.

Let us explain the idea of proof briefly. Besides tools invented by Larsen, a new fact which is used frequently in this
paper is the following: if a compact Lie group $G$ is acceptable (resp. strongly acceptable), then $Z_{G}(A)$ is also
acceptable (resp. strongly acceptable) for any closed abelian subgroup $A$ of $G$. With this fact, we deduce the strong
acceptability of $\SO(4)$ from that of $\G_{2}$. We show unacceptability of the following groups: \[\SO(6)\ (\textrm{due to
Weidner (\cite{Weidner})}),\] \[\Sp(1)^{m}/\langle(-1,\dots,-1)\rangle\ (m\geq 3),\] \[\Sp(1)^{3}/\langle(1,-1,-1),(-1,1,-1)
\rangle.\] These are new findings after \cite{Larsen1}, \cite{Larsen2}. With the above fact, we show many compact Lie groups
are unacceptable by reducing to the unacceptability of these three particular groups.

The element-conjugacy vs. global conjugacy question is important for determining the monodromy group of a Galois homomorphism
(\cite{Arthur},\cite{Chenevier}). It is also closely related to the multiplicity one problem in automorphic form theory
(\cite{Arthur},\cite{Wang}) and pseudocharacters in number theory (\cite{Bockle-Harris-Khare-Thorne},\cite{Weidner}). The
latter is an important tool for proving modularity lifting theorems  (\cite{Taylor},\cite{Bockle-Harris-Khare-Thorne}). For
example, for any odd prime $p$ it is known that the cuspidal spectrum of $\SL_{p}(\mathbb{A})$ fails to have the multiplicity
one property as reflected by the existence of non-globally conjugate but element-conjugate homomorphisms from
$(\mathbb{Z}/p\mathbb{Z})^{2}$ to $\PSL_{p}(\mathbb{C})$ (\cite{Blasius}, \cite{Lapid}). In general, it is expected that if
the Langlands $L$-group $^{L}G=\hat{G}\rtimes\Gal(E/F)$ is unacceptable, then the cuspidal spectrum of $G(\mathbb{A}_{F})$
fails to have the multiplicity one property. The $L$-group of split $\SO_{4}$ over a number field $F$ is $\SO_{4}(\mathbb{C})$.
It is known that $\SO_{4}(\mathbb{A}_{F})$ does have the multiplicity one property (\cite{Lapid-Mao}), which is consistent with
the acceptability of $\SO_{4}$ shown in this paper. For more discussions on the relation between acceptability, pseudocharacters
and invariant function rings, the reader is encouraged to read the last section. We show that there are invariant functions
on $\SO_{4}(\mathbb{C})^{2}$ which are not generated by 1-argument invariants, though the group $\SO_{4}(\mathbb{C})$
is acceptable.

\smallskip

\noindent{\it Notation and conventions.} Let $G$ be a group. Write $Z(G)$ for the subgroup of central elements of $G$,
write $G_{\der}=[G,G]$ for the derived subgroup of $G$; write $G^{0}$ for the neutral subgroup of $G$ when $G$ is
a Lie group. Write $Z_{G}(A)=\{g\in G:gxg^{-1}x^{-1}=1 (\forall x\in A)\}$ for the centralizer of a subset $A\subset G$,
and write $N_{G}(A)=\{g\in G:gAg^{-1}=A\}$ for the normalizer of $A$ in $G$. When $A=\{x\}$ consists of a single element,
we also write $G^{x}$ or $Z_{G}(x)$ for $Z_{G}(\{x\})$. For any element $g\in G$, write $\Ad(g): G\rightarrow G$ for
the automorphism of $G$ defined by $\Ad(g)(x)=gxg^{-1}$ ($\forall x\in G$). For two elements $x,y\in G$, we write
$x\sim y$ if they are conjugate in $G$, i.e., if there exists $g\in G$ such that $y=gxg^{-1}$; write $x\sim_{H} y$
if they are conjugate by an element in a subgroup $H$.

For two closed subgroups $G_{1},G_{2}$ of a Lie group $G$, we write $G=G_{1}\cdot G_{2}$ if $g_{1}g_{2}=g_{2}g_{1}$
($\forall(g_{1},g_{2})\in G_{1}\times G_{2}$), any element $g\in G$ admits a decomposition $g=g_{1}g_{2}$ for some pair
$(g_{1},g_{2})\in G_{1}\times G_{2}$, and $G_{1}\cap G_{2}$ is a finite subgroup of $Z(G)$. We write $G=G_{1}\times G_{2}$
if furthermore $G_{1}\cap G_{2}=1$. In this case we say that $G$ is the direct product of $G_{1}$ and $G_{2}$.

Write $I_{n}$ for an $n\times n$ identity matrix. Let $\U(n)$ (resp. $\O(n)$) denote the unitary group (resp. orthogonal
group) of degree $n$. Then, $\SU(n):=\{X\in\U(n):\det X=1\}$, $\PU(n):=\U(n)/Z(\U(n))$, $\PSU(n):=\SU(n)/Z(\SU(n))$,
$\SO(n):=\{X\in\O(n):\det X=1\}$, $\PO(n):=\O(n)/Z(\O(n))$, $\PSO(n):=\SO(n)/Z(\SO(n))$.

Let $V$ be an $n$-dimensional Euclidean space with an orthonormal basis $\{e_{1},e_{2},\dots,e_{n}\}$. Then we have
the Clifford algebra $\Clif(V)$ (which is an associative algebra with unity) generated by $V$ subject to the relations
\[e_{i}e_{j}+e_{j}e_{i}=-2\delta_{i,j}\] where $\delta_{i,j}=\left\{\begin{array}{cc}1\textrm{ if }i=j;\\0\textrm{ if }
i\neq j.\end{array}\right.$ In $\Clif(V)$, we have the pin group $\Pin(V)$ generated by $\{u\in V: |u|=1\}$, which
is regarded as a subgroup of the group of invertible elements in $\Clif(V)$. Note that $u^2=-1$ for any element
$u\in V$ with $|u|=1$. The spin group $\Spin(V)$ is generated by $\{uv:u,v\in V, |u|=|v|=1\}$. The pin group $\Pin(V)$
has two connected components; $\Spin(V)$ is its neutral subgroup, hence connected. In $\Pin(V)$, write
\[c=e_{1}\cdots e_{n}.\] Then, $c\in\Spin(V)$ if and only if $n$ is even, and $c$ is a central element of $\Spin(V)$
when $n$ is even. When $\dim V\geq 3$, $\Spin(V)$ is also simply-connected. Put $n=\dim V$. Then, let $\Pin(n),\Spin(n)$
denote for $\Pin(V),\Spin(V)$ respectively. When $4|n$, define the half-spin group by $\HSpin(n)=\Spin(n)/\langle c\rangle$.

Let $\mathbb{H}=\{a+b\mathbf{i}+c\mathbf{j}+d\mathbf{k}:a,b,c,d\in\mathbb{R}\}$ be the quaternion algebra, where the
generators $\mathbf{i},\mathbf{j},\mathbf{k}$ are subject to relations $\mathbf{i}^2=\mathbf{j}^2=\mathbf{k}^2=-1$,
$\mathbf{i}\mathbf{j}=-\mathbf{j}\mathbf{i}=\mathbf{k}$, $\mathbf{k}\mathbf{i}=-\mathbf{i}\mathbf{k}=\mathbf{j}$,
$\mathbf{j}\mathbf{k}=-\mathbf{k}\mathbf{j}=\mathbf{i}$. In $\mathbb{H}$, we have the conjugation
$\bar{}:\mathbb{H}\rightarrow\mathbb{H}$ (which is an anti-automorphism of $\mathbb{H}$ as a $\mathbb{R}$-algebra)
defined by \[\overline{a+b\mathbf{i}+c\mathbf{j}+d\mathbf{k}}=a-b\mathbf{i}-c\mathbf{j}-d\mathbf{k}.\] Then, the
symplectic group $\Sp(n)$ is defined by \[\Sp(n)=\{X\in M_{n}(\mathbb{H}):X\bar{X}^{t}=I\},\] where $X^{t}$ means
the transposition of a matrix $X$. In particular, \[\Sp(1)=\{a+b\mathbf{i}+c\mathbf{j}+d\mathbf{k}\in\mathbb{H}:
a^2+b^2+c^2+d^2=1\}.\] The projective symplectic group of degree $n$ is $\PSp(n):=\Sp(n)/\langle-I\rangle$.

Let $\E_{6}$ or $\E_{6}^{sc}$ denote a connected and simply-connected compact simple Lie group of type $\mathbf{E}_6$,
and let $\E_{6}^{ad}$ denote an adjoint type  connected compact simple Lie group of type $\mathbf{E}_6$. Similarly,
we have notations $\E_{7}$, $\E_{7}^{sc}$, $\E_{7}^{ad}$ for type $\mathbf{E}_7$. Let $\E_{8},\F_{4},\G_{2}$ denote
a connected compact simple Lie group of type $\mathbf{E}_8,\mathbf{F}_4,\mathbf{G}_2$ respectively, which are both
simply-connected and of adjoint type.

For a positive integer $m$, write $\mu_{m}=\{\lambda\in\mathbb{C}:\lambda^{m}=1\}$, which is a cyclic group of
order $m$. For simplicity we also let $\mu_{m}$ denote $\{\lambda I_{n}:\lambda^{m}=1\}$ for any $n\geq 1$, which
is an order $m$ central subgroup of $\U(n)$ and is contained in $\SU(n)$ if and only if $m|n$.

Write $[x]$ for the class of an element $x\in G$ in a quotient group $G/N$; and write $[x,y]=xyx^{-1}y^{-1}$ for the
commutator of two elements $x,y\in G$.

\smallskip

\noindent{\it Acknowledgements.} I would like to thank Xinwen Zhu for bringing up this question to my attention,
and to thank Wee Teck Gan, Ga\"etan Chenevier and Matthew Weidner for helpful communications. I would like to
thank the referees for very helpful comments and suggestions. This research is partially supported by the NSFC
Grant 11971036.

\section{A classification of acceptable connected compact Lie groups}\label{S:acceptC}

\subsection{Reductions}

The following lemma \ref{L:product} is obvious.

\begin{lemma}\label{L:product}
If a compact Lie group $G=G_{1}\cdot G_{2}$ is acceptable (resp. strongly acceptable), then both $G_1$ and $G_2$ are.
If this is a direct product, then the converse also holds true.
\end{lemma}

The following lemma \ref{L:Larsen} is a variance of \cite[Prop. 1.4]{Larsen1}. For readers' convenience we give a proof,
which is identical to the proof of Larsen in \cite{Larsen1}.

\begin{lemma}[]\label{L:Larsen}
Let $H$ be a closed subgroup of a compact Lie group $G$ such that $G=Z(G)\cdot H$. Then $G$ is acceptable (resp. strongly
acceptable) if and only if $H$ is. In particular, a connected compact Lie group $G$ is acceptable (resp. strongly
acceptable) if and only if $G_{\der}$ is.
\end{lemma}

\begin{proof}
We first show for strong acceptability in the first statement. Assume that $H$ is strongly acceptable.
Let $\Gamma$ be a group and $\phi_{1},\phi_{2}:\Gamma\rightarrow G$ be two element-conjugate homomorphisms.
Let $\pi:G\rightarrow G/Z(G)$ be the projection map. Define\[\Gamma'=\{(\gamma,x)\in\Gamma\times H:
\pi(\phi_{2}(\gamma))=\pi(x)\}.\] For any $y=(\gamma,x)\in\Gamma'$, put $\eta(y)=\phi_{2}(\gamma)x^{-1}
\in Z(G)$. Then, $\eta:\Gamma'\rightarrow Z(G)$ is a homomorphism. Define $\tilde{\phi}_{2}(\gamma,x)=x=
\phi_{2}(\gamma)\eta(\gamma,x)^{-1}$ and $\tilde{\phi}_{1}(\gamma,x)=\phi_{1}(\gamma)\eta(\gamma,x)^{-1}$
($\forall(\gamma,x)\in\Gamma'$). Then, $\tilde{\phi}_{1}$ and $\tilde{\phi}_{2}$ are element-conjugate as
$\phi_{1}$ and $\phi_{2}$ are. Since $H$ is a normal subgroup of $G$, we have $\Im\tilde{\phi}_{1}\subset H$
as $\Im\tilde{\phi}_{2}\subset H$. Then, $\tilde{\phi}_{1},\tilde{\phi}_{2}:\Gamma'\rightarrow H$ are element-conjugate
homomorphisms. As $H$ is strongly acceptable, there exists $g\in H$ such that $\tilde{\phi}_{1}=\Ad(g)\circ
\tilde{\phi}_{2}$. Since $\Gamma'\twoheadrightarrow\Gamma$, then $\phi_{1}=\Ad(g)\circ\phi_{2}$. Thus, $G$ is
strongly acceptable.

The proof of the acceptability in the first statement is shown in the same line as for strong acceptability.
Just note that $\ker(H\rightarrow G/Z(G))=H\cap Z(G)$ is a finite group by assumption. Then $\Gamma'$ is a finite
group whenever $\Gamma$ is.

The second statement follows from the first since $G=Z(G)\cdot G_{\der}$ and $Z(G)\cap G_{\der}$ is a finite group
for any connected compact Lie group $G$.
\end{proof}

The following lemma \ref{L:centralizer2} is crucial for many arguments in this paper.

\begin{lemma}\label{L:centralizer2}
If a compact Lie group $G$ is acceptable (resp. strongly acceptable), then $Z_{G}(A)$ is also acceptable (resp. strongly
acceptable) for any closed abelian subgroup $A$ of $G$.
\end{lemma}

\begin{proof}
We first give the proof for strong acceptability. Assume that $G$ is strongly acceptable. Write $H=Z_{G}(A)$. Let $\Gamma$
be a group, and $\phi_{1},\phi_{2}:\Gamma\rightarrow H$ be two element-conjugate homomorphisms. Put $\Gamma'=\Gamma\times A$.
Define $\phi'_{i}: \Gamma'\rightarrow G$ ($i=1,2$) by $$\phi'_{i}(\gamma,a)=\phi_{i}(\gamma)a,\ \forall\gamma\in\Gamma,
a\in A.$$ Apparently, $\phi'_{1}$ and $\phi'_{2}$ are element-conjugate homomorphisms. Since $G$ is strongly acceptable,
there exists $g\in G$ such that $$\phi'_{2}(\gamma,a)=\Ad(g)(\phi'_{1}(\gamma,a)),\ \forall(\gamma,a)\in\Gamma'.$$ Applying
to $\gamma=1$ and $a\in A$, we get $g\in Z_{G}(A)=H$. Applying to $a=1$ and $\gamma\in\Gamma$, we get $\phi_{2}(\gamma)=
\Ad(g)(\phi_{1}(\gamma))$ for any $\gamma\in\Gamma$. This just means: $\phi_1$ and $\phi_2$ are globally conjugate. Thus,
$H$ is strongly acceptable.

For the acceptability, there exists $m>1$ such that $Z_{G}(A(m))=Z_{G}(A)$, where $A(m)=\{x\in A: x^{m}=1\}.$ Using $A(m)$
instead of $A$ in the above argument, the proof proceeds the same.
\end{proof}

It is clear that the following variance of Lemma \ref{L:centralizer2} holds true: if a Lie group $G$ is acceptable,
then $Z_{G}(A)$ is also acceptable for any finite subgroup $A$ of $G$; if a Lie group $G$ is strongly acceptable, then
$Z_{G}(A)$ is also strongly acceptable for any subgroup $A$ of $G$. In this paper, we only need the case when $A$ is a
closed abelian subgroup, so we state Lemma \ref{L:centralizer2} as that.

\smallskip

While applying Lemma \ref{L:centralizer2}, we frequently calculate the fixed point subgroup of an automorphism
or the centralizer of a torus, for which the following several well-known results are very useful.

\begin{theorem}[\cite{Steinberg}]\label{T:Steinberg}
Let $G$ be a connected and simply-connected compact Lie group and $\theta$ be an automorphism of $G$. Then the fixed
point subgroup $G^{\theta}$ is connected.
\end{theorem}

\begin{theorem}[\cite{Knapp}]\label{T:Knapp}
Let $A$ be a torus in a connected compact Lie group $G$. Then $Z_{G}(A)$ is connected.
\end{theorem}

\begin{theorem}[\cite{Borel-Friedman-Morgan}]\label{L:Borel}
If $L$ is a Levi subgroup of a connected and simply-connected compact Lie group $G$, then $L_{\der}$ is simply-connected.
\end{theorem}

In \cite{Borel-Friedman-Morgan}, one can also find useful methods to calculate the fixed point subgroup $G^{\theta}$ (for
$\theta$ an automorphism) and useful characterization of $L_{\der}$ (for $L$ a Levi subgroup) when $G$ is non-simply
connected.

\smallskip

By the following lemma \ref{L:acceptable1}, to check strong acceptability of a compact Lie group it suffices to check
for continuous homomorphisms from compact Lie groups. The proof of Lemma \ref{L:acceptable1} below is suggested by a
referee, which improves the statement of the lemma and enables us to change the definition of strong acceptability
by allowing homomorphisms from any group. Previously, we consider only continuous homomorphisms from compact topological
groups while defining strong acceptability.

\begin{lemma}\label{L:acceptable1}
For a compact Lie group $G$ to be strongly acceptable it is necessary and sufficient that for any compact Lie group
$H$ and every pair of element-conjugate continuous homomorphisms $\phi_1,\phi_2:H\rightarrow G$, $\phi_1$ and $\phi_2$
are globally conjugate.
\end{lemma}

\begin{proof}
The necessarity is trivial. We prove the sufficiency. Suppose the condition holds for homomorphisms from compact Lie
groups. Let $\Gamma$ be a group and $\phi_1,\phi_2:\Gamma\rightarrow G$ be two element-conjugate homomorphisms. Firstly,
$\ker\phi_1=\ker\phi_2$. Considering $\Gamma/\ker\phi_{1}$ instead, we may assume that $\phi_1,\phi_2$ are injective.
Embed $\Gamma$ into $G\times G$ via $(\phi_1,\phi_2)$ and let $H$ be the closure of the image of this embedding. Then,
$H$ is a compact Lie group. Write $\phi=(\phi_1,\phi_2):\Gamma\rightarrow H$ and let $\pr_1,\pr_2:H\rightarrow G$ be
the projection from $H$ to the first and the second component of $G\times G$ respectively. Then, $\phi_1=
\pr_1\circ\phi$ and $\phi_2=\pr_2\circ\phi$. It suffices to show that: $\pr_1$ and $\pr_2$ are element-conjugate,
i.e., $x\sim y$ for any element $(x,y)\in H$. By definition $(x,y)$ is the limit of a sequence
$\{(g_{n},a_{n}g_{n}a_{n}^{-1})\}_{n=1}^{\infty}$. Replacing $\{(g_{n},a_{n}g_{n}a_{n}^{-1})\}_{n=1}^{\infty}$ by a
sub-sequence if necessary, we may assume that both sequences $\{g_{n}\}_{n=1}^{\infty}$ and $\{a_{n}\}_{n=1}^{\infty}$
converge. Put $g=\lim_{n\to\infty}g_{n}$ and $a=\lim_{n\to\infty}a_{n}.$ Then, $x=g$ and $y=aga^{-1}$. Thus, $x\sim y$.
\end{proof}

\subsection{Examples of acceptable compact Lie groups}

\begin{proposition}\label{P:acceptable1}
Any group $G$ in the following list is strongly acceptable: \begin{itemize}
\item[(i)]$\U(n)$;
\item[(ii)]$\SU(n)$;
\item[(iii)]$\Sp(n)$;
\item[(iv)]$\O(n)$;
\item[(v)]$\SO(2n+1)$;
\item[(vi)]$\G_2$;
\item[(vii)]$\SO(4)$.
\end{itemize}
\end{proposition}

\begin{proof}
The strong acceptability of $\U(n)$ follows from the character theory of representations of compact Lie groups, which
implies the strong acceptability of $\SU(n)$. The strong acceptability of $\Sp(n), \O(n), \G_{2}$ is shown in \cite{Griess}
and \cite{Larsen1}. Due to $\O(2n+1)=\SO(2n+1)\times\{\pm{I}\}$, it follows that $\SO(2n+1)$ is also strongly acceptable.
Take an involution $\theta\in\G_2$. Then, $\G_{2}^{\theta}\cong\SO(4)$ (\cite[Table2]{Huang-Yu}). By Lemma
\ref{L:centralizer2}, $\SO(4)$ is strongly acceptable as $\G_2$ is.
\end{proof}

Besides $\SO(4)$, the strong acceptability of all other groups in Prop. \ref{P:acceptable1} are well-known. The only new
finding of ours is the strong acceptability of $\SO(4)$.

\subsection{Examples of unacceptable compact Lie groups}

\begin{example}[Weidner, \cite{Weidner}]\label{E:SU4}
Let $G=\SU(4)/\langle-I\rangle$ or $\SU(4)/\langle\mathbf{i}I\rangle$, and let $\Gamma=(\mathbb{Z}/4\mathbb{Z})^{2}$
with two generators $\gamma_1,\gamma_2$. Define $\phi,\phi':\Gamma\rightarrow\SU(4)$ by \[\phi(\gamma_1)=
\diag\{1,1,\mathbf{i},-\mathbf{i}\},\quad\phi(\gamma_2)=\diag\{1,\mathbf{i},1,-\mathbf{i}\}\] and $\phi'(\gamma)=
\overline{\phi(\gamma)}$ ($\forall\gamma\in\Gamma$). Write $\pi:\SU(4)\rightarrow G$ for the natural projection. Set
$\rho=\pi\circ\phi$ and $\rho'=\pi\circ\phi'$. In the below, we show that $\rho$ and $\rho'$ are element-conjugate,
but not globally conjugate. Thus, $G$ is unacceptable.
\end{example}

\begin{proof}
Write $\gamma=\gamma_1^{a}\gamma_2^{b}$ ($a,b\in\{0,1,2,3\}$) for a general element of $\Gamma$. Then, $$\phi(\gamma)
=\diag\{1,\mathbf{i}^{b},\mathbf{i}^{a},(-\mathbf{i})^{a+b}\}.$$ When one of $a,b,a+b$ is a multiple of $4$, we have
$\phi'(\gamma)=\overline{\phi(\gamma)}\sim\phi(\gamma)$; when one of $a-2,b-2,a+b-2$ is a multiple of $4$, we have
$\phi'(\gamma)=\overline{\phi(\gamma)}\sim-\phi(\gamma)$. This covers all $\gamma\in\Gamma$. Then, $\rho'(\gamma)
\sim\rho(\gamma)$ ($\forall\gamma\in\Gamma$). Thus, $\rho$ and $\rho'$ are element-conjugate.

Suppose $\rho$ and $\rho'$ are globally conjugate. Then, there exists $g\in\SU(4)$ such that $\overline{\phi(\gamma_1)}
=\mathbf{i}^{t_{1}}g\phi(\gamma_1)g^{-1}$ and $\overline{\phi(\gamma_2)}=\mathbf{i}^{t_{2}}g\phi(\gamma_2)g^{-1},$
where $t_1,t_2\in\{0,1,2,3\}$. Since each of $\phi(\gamma_1),\phi(\gamma_2),\overline{\phi(\gamma_1)},
\overline{\phi(\gamma_{2})}$ has eigenvalue $1$ twice and has eigenvalues $\mathbf{i},-\mathbf{i}$ each once, then
$t_{1}=t_{2}=0$. Since $\overline{\phi(\gamma_{1})}^{2}=\phi(\gamma_{1})^{2}$ and $\overline{\phi(\gamma_{2})}^{2}
=\phi(\gamma_{2})^{2}$, then $g$ commutes with $\phi(\gamma_{1})^{2}=(1,1,-1,-1)$ and $\phi(\gamma_{2})^{2}=
(1,-1,1,-1)$. Then, $g$ is a diagonal matrix. Thus, $g$ commutes with $\phi(\gamma_1)$. Hence, $\phi(\gamma_1)=
\mathbf{i}^{t_{1}}g\phi(\gamma_1)g^{-1}=\overline{\phi(\gamma_1)}\neq\phi(\gamma_1)$, which is a contradiction.
Thus, $\rho$ and $\rho'$ are not globally conjugate.
\end{proof}

Due to $\SO(6)\cong\SU(4)/\langle-I\rangle$, $\SO(6)$ is also unacceptable. The unacceptability of $\SO(6)$ is first
shown by Matthew Weidner (\cite{Weidner}). Actually, the above example of element-conjugate but not globally conjugate
homomorphisms from $(\mathbb{Z}/4\mathbb{Z})^{2}$ to $\SU(4)/\langle-I\rangle$ is the counter-part of the example of
a pair of homomorphisms from $(\mathbb{Z}/4\mathbb{Z})^{2}$ to $\SO(6)$ constructed by Matthew Weidner in \cite{Weidner}.

\smallskip

\begin{example}\label{E:Sp3-2}
For any $m\geq 3$, let $G=\Sp(1)^{m}/\langle(-1,\dots,-1)\rangle$. Let $\Gamma=(\mathbb{Z}/4\mathbb{Z})^{2}$ with two
generators $\gamma_1,\gamma_2$. Define $\phi,\phi':\Gamma\rightarrow\Sp(1)^{m}$ by \[\phi(\gamma_1)\!=\!\phi'(\gamma_{1})
\!=\!(1,\dots,1,\mathbf{i},\mathbf{i}),\ \phi(\gamma_2)\!=\!(\mathbf{i},\dots,\mathbf{i},1,\mathbf{i}),
\ \phi'(\gamma_2)\!=\!(\mathbf{i},\dots,\mathbf{i},1,-\mathbf{i}).\] Let $\pi:\Sp(1)^{m}\longrightarrow
\Sp(1)^{m}/\langle(-1,\dots,-1)\rangle=G$ be the natural projection. Set $\rho=\pi\circ\phi$, $\rho'=\pi\circ\phi'.$
In the below, we show that $\rho$ and $\rho'$ are element-conjugate, but not globally conjugate. Thus, $G$ is
unacceptable.
\end{example}

\begin{proof}
Put $z=(-1,\dots,-1)\in Z(\Sp(1)^{m})$. Write $\gamma=\gamma_1^{a}\gamma_2^{b}$ ($a,b\in\{1,2,3,4\}$) for a general
element of $\Gamma$. Then, \[\phi(\gamma)=(\mathbf{i}^{b},\dots,\mathbf{i}^{b},\mathbf{i}^{a},\mathbf{i}^{a+b})
\textrm{ and }\phi'(\gamma)=(\mathbf{i}^{b},\dots,\mathbf{i}^{b},\mathbf{i}^{a},(-1)^{b}\mathbf{i}^{a+b}).\] When
$b$ is even, we have $\phi'(\gamma)=\phi(\gamma)$; when $b$ is odd and $a$ is even, we have $\phi'(\gamma)\sim
\phi(\gamma)$; when $b$ and $a$ are both odd, we have $\phi'(\gamma)\sim z\phi(\gamma)$. In any case we have
$\rho'(\gamma)\sim\rho(\gamma)$. Thus, $\rho$ and $\rho'$ are element-conjugate.

Suppose $\rho$ and $\rho'$ are globally conjugate. We will reach a contradiction by projecting to the last three
coordinates: $G\rightarrow\Sp(1)^{3}/\langle(-1,-1,-1)\rangle$. So, we may assume that $m=3$. There exists $g=
(g_1,g_2,g_3)\in\Sp(1)^{3}$ such that: for each $\gamma\in\Gamma$, there exists $t\in\{0,1\}$ such that
$\phi'(\gamma)=z^{t}g\phi(\gamma)g^{-1}$. Write $$\phi'(\gamma_1)=z^{t_{1}}g\phi(\gamma_1)g^{-1}\textrm{ and }
\phi'(\gamma_2)=z^{t_{2}}g\phi(\gamma_2)g^{-1}$$ where $t_1,t_2\in\{0,1\}$. Then, for each $j\in\!\{1,2,3\}$,  $g_{j}\mathbf{i}g_{j}^{-1}=\eta_{j}\mathbf{i}$ for some $\eta_{j}=\pm{1}$. From $\phi'(\gamma_1)=\phi(\gamma_1)
=(1,\mathbf{i},\mathbf{i})$, we get $t_1=0$ and $\eta_2=\eta_3=1$. From $\phi(\gamma_2)=(\mathbf{i},1,\mathbf{i})$
and $\phi'(\gamma_2)=(\mathbf{i},1,-\mathbf{i})$, we get $t_2=0$, $\eta_1=1$, $\eta_3=-1$. It is a
contradiction that $1=\eta_3=-1$. Thus, $\rho$ and $\rho'$ are not globally conjugate.
\end{proof}

\smallskip

\begin{example}\label{E:3A1-2}
Let $G=\Sp(1)^{3}/\langle(1,-1,-1),(-1,1,-1)\rangle$. Write $\eta=\frac{1+\mathbf{i}}{\sqrt{2}}$. Let $\Gamma\subset G$
be a finite subgroup generated by \[ [(\mathbf{j},\eta,\eta)],\ [(\eta,\mathbf{j},\eta)],\ [(\eta,\eta,\mathbf{j})],
\ [(\mathbf{i},\mathbf{i},\mathbf{i})].\] Write $\gamma_1=[(\mathbf{j},\eta,\eta)]$, $\gamma_{2}=[(\eta,\mathbf{j},
\eta)]$, $\gamma_3=[(\eta,\eta,\mathbf{j})]$, $\gamma_0=[(\mathbf{i},\mathbf{i},\mathbf{i})]$ and $z_{0}=[(-1,-1,-1)]$.
Put $\Gamma_{0}=\langle z_{0},\gamma_1,\gamma_2,\gamma_3\rangle$. Let $\rho:\Gamma\rightarrow G$ be the inclusion and
let $\rho':\Gamma\rightarrow G$ be defined by $\rho'|_{\Gamma_0}=\id$ and $\rho'(\gamma)=z_{0}\gamma$ ($\forall\gamma
\in\Gamma-\Gamma_{0}$). Then, $\rho$ and $\rho'$ are element-conjugate homomorphisms, but not globally conjugate. Thus,
$G$ is unacceptable.
\end{example}

\begin{proof}
Put $y_1=[(1,\mathbf{i},\mathbf{i})]$, $y_{2}=[(\mathbf{i},1,\mathbf{i})]$, $y_3=[(\mathbf{i},\mathbf{i},1)]$ and
$\Gamma_{1}=\langle z_0,y_1,y_2,y_3\rangle$. It is clear that $\Gamma_1$ is a normal subgroup of $\Gamma$. Note that
$y_{k}^{2}=1$ ($k=1,2,3$) and $y_{1}y_{2}y_{3}=z_{0}$. Thus, $\Gamma_1\cong(\mathbb{Z}/2\mathbb{Z})^{3}$.
Since $\gamma_{k}^{2}=[\gamma_{k+1},\gamma_{k+2}]=z_{0}y_{k}$ ($k=1,2,3$, here we let $\gamma_{4}=\gamma_1$
and $\gamma_{5}=\gamma_{2}$), we see that: $\Gamma_{0}/\Gamma_{1}\cong(\mathbb{Z}/2\mathbb{Z})^{3}$. Since
$\gamma_{0}^{2}=z_{0}$ and $\gamma_{0}\gamma_{k}\gamma_{0}^{-1}=z_{0}\gamma_{k}$ ($k=1,2,3$), we get
$\Gamma/\Gamma_{0}\cong\mathbb{Z}/2\mathbb{Z}$. This explains that $\rho'$ is a well-defined homomorphism.

Any element $\gamma\in\Gamma-\Gamma_{0}$ is of the form \[\gamma=\gamma_{0}\gamma_{1}^{t_{1}}\gamma_{2}^{t_{2}}
\gamma_{3}^{t_{3}}y_{1}^{s_{1}}y_{2}^{s_{2}}y_{3}^{s_{3}}\] with $t_{1},t_{2},t_{3},s_{1},s_{2},s_{3}\in\{0,1\}$.
Write $\gamma=[(x_1,x_2,x_3)]$ ($x_{k}\in\Sp(1)$). If some $t_{k}=1$ ($k=1,2,3$), then $x_{k}\sim\mathbf{j}\sim
\mathbf{i}$. If all $t_{1},t_{2},t_{3}=0$, then at least one $x_{k}\sim\mathbf{i}$ since each $x_{k}\in\{1,-1,\mathbf{i},
-\mathbf{i}\}$ and $x_{1}x_{2}x_{3}=\pm{\mathbf{i}}$. Without loss of generality we assume that $x_{1}\sim\mathbf{i}$.
Then, \[z_{0}\gamma=[(-x_1,x_2,x_3)]\sim[(x_1,x_2,x_3)]=\gamma.\] Thus, $\rho$ and $\rho'$ are element-conjugate.

Suppose that $\rho$ and $\rho'$ are globally conjugate, i.e., there exists $g\in G$ such that $g\rho(\gamma)g^{-1}=
\rho'(\gamma)$ ($\forall\gamma\in\Gamma$). Let \[\pi:G\rightarrow\Sp(1)^{3}/\langle(-1,1,1),(1,-1,1),(1,1,-1)\rangle
\cong\SO(3)^{3}\] be the natural projection. Then, $\pi(\rho(\gamma))=\pi(\rho'(\gamma))$ ($\forall\gamma\in\Gamma$).
Thus, $\pi(g)$ commutes with $\pi(\Gamma)$. This implies that: $g\in\langle[(\mathbf{i},1,1)],[(1,\mathbf{i},1)],
[(1,1,\mathbf{i})]\rangle$. Then, \[\gamma_{0}=g\gamma_{0}g^{-1}=g\rho(\gamma_{0})g^{-1}=\rho'(\gamma_0)=
z_{0}\gamma_{0}\neq\gamma_{0},\] which is a contradiction. Thus, $\rho$ and $\rho'$ are not globally conjugate.
\end{proof}

We deduce some immediate consequences of Example \ref{E:Sp3-2}.

\begin{example}[Chenevier-Gan, \cite{Chenevier-Gan}]\label{E:unacceptable-ChenevierGan}
The group $\Spin(7)$ is unacceptable.
\end{example}

\begin{proof}
Take $x=e_1e_2e_3e_4\in\Spin(7)$. Then, $$\Spin(7)^{x}=\Spin(4)\cdot\Spin(3)\cong\Sp(1)^{3}/\langle(-1,-1,-1)
\rangle.$$ By Example \ref{E:Sp3-2}, $\Sp(1)^{3}/\langle(-1,-1,-1)\rangle$ is unacceptable. Then by Lemma
\ref{L:centralizer2}, so is $\Spin(7)$.
\end{proof}

The unacceptability of $\Spin(7)$ is first shown by Ga\"etan Chenevier and Wee Teck Gan (\cite{Chenevier-Gan}). In
\cite{Larsen2} Michael Larsen made a mistake while proving that ``$\Spin(7)$ is acceptable". The author quoted this wrong
result in the first version of this paper. After posting this version on arXiv, Wee Teck Gan kindly showed the author a
letter of him and Chenevier to Larsen which contains a concrete pair of element-conjugate but not globally conjugate
homomorphisms from $(\mathbb{Z}/4\mathbb{Z})^{2}$ to $\Spin(7)$. More than letting the author realize this mistake,
Chenevier-Gan 's construction also lets him realize that $\Sp(1)^{m}/\langle(-1,\dots,-1)\rangle$ might be unacceptable
and inspires the construction in Example \ref{E:Sp3-2}.

The following question is interesting in that its resolution might help identify Langlands parameters of automorphic
representations of $\GSp_{6}$ of Artin type.

\begin{question}\label{Q:Spin7}
Can one classify all pairs of element-conjugate but not globally conjugate homomorphisms from a finite group to $\Spin(7)$?
\end{question}

The unacceptability of $\Spin(8)$ was first shown by Larsen using homomorphisms from a finite simple group
$\SL(3,\mathbb{Z}/2\mathbb{Z})$ (\cite[Prop. 2.5]{Larsen2}). In the following Example \ref{E:Spin8} we give
a much easier proof.

\begin{example}\label{E:Spin8}
The group $\Spin(8)$ is unacceptable.
\end{example}

\begin{proof}
Take $x=e_1e_2e_3e_4\in\Spin(8)$. Then, $$\Spin(8)^{x}=\Spin(4)\cdot\Spin(4)\cong\Sp(1)^{4}/\langle(-1,-1,-1,-1)\rangle.$$
By Example \ref{E:Sp3-2}, $\Sp(1)^{4}/\langle(-1,-1,-1,-1)\rangle$ is unacceptable. Then by Lemma \ref{L:centralizer2},
so is $\Spin(8)$.
\end{proof}

The following Example \ref{E:Ap-1} is well-known.

\begin{example}\label{E:Ap-1}
For any odd prime $p$, the group $\SU(p)/\langle\omega_{p}I\rangle$ is unacceptable, where $\omega_{p}=e^{\frac{2\pi i}{p}}$.
\end{example}

\begin{proof}
Put $$A_{p}=\diag\{1,\omega_{p},\dots,\omega_{p}^{p-1}\}$$ and \[B_{p}=\left(\begin{array}{ccccc}0&1&0&\ldots&0\\0&0&1&\ldots
&0\\0&0&0&\ldots&0\\\vdots&\vdots&\vdots&\ddots&\vdots\\1&0&0&\ldots&0\\\end{array}\right).\] Then, $A_{p},B_{p}\in\SU(p)$,
$(A_{p})^{p}=(B_{p})^{p}=I$ and $A_{p}B_{p}A_{p}^{-1}B_{p}^{-1}=\omega_{p}^{-1}I$. Moreover, $A_{p}^{i}B_{p}^{j}\sim A_{p}$ for
any pair $(i,j)\in\{0,1,\dots,p-1\}^{2}$ except when $i=j=0$.

Put $\Gamma=(\mathbb{Z}/p\mathbb{Z})^{2}$ with two generators $\gamma_{1},\gamma_{2}$. Define homomorphism $\rho,\rho':
\Gamma\rightarrow\SU(p)/\langle\omega_{p}I\rangle$ by $\rho(\gamma_1)=\rho'(\gamma_1)=[A_{p}]$, $\rho(\gamma_2)=[B_{p}]$
and $\rho'(\gamma_2)=[B_{p}^{-1}]$. By statements in the first paragraph, $\rho,\rho'$ are well-defined homomorphisms and
they are element-conjugate.

Suppose that $\rho,\rho'$ are globally conjugate. Then there exists $g\in\SU(p)$ such that $gA_{p}g^{-1}=\omega_{p}^{t_1}A_{p}$
and $gB_{p}g^{-1}=\omega_{p}^{t_2}B_{p}^{-1}$ for some $t_{1},t_{2}\in\mathbb{Z}$. Then, \[\omega_{p}^{-1}I=[A_{p},B_{p}]=
[gA_{p}g^{-1},gB_{p}g^{-1}]=[\omega_{p}^{t_1}A_{p},\omega_{p}^{t_2}B_{p}^{-1}]=\omega_{p}I.\] Thus, $\omega_{p}^{2}=1$,
which contradicts to $p$ being an odd prime.
\end{proof}

\begin{example}\label{E:Ap-2}
Assume that $k\geq 1$ and $s_{1},\dots,s_{k}\in\{1,2,\dots,p-1\}$. Then, $G=\SU(p)^{k}/\langle(\omega_{p}^{s_{1}}I,\cdots,
\omega_{p}^{s_{k}}I)\rangle$ is unacceptable.
\end{example}

\begin{proof}
Put $\Gamma=(\mathbb{Z}/p\mathbb{Z})^{2}$ with two generators $\gamma_{1},\gamma_{2}$. Define homomorphism $\rho,\rho':
\Gamma\rightarrow G$ by \[\rho(\gamma_1)=\rho'(\gamma_1)=[\underbrace{(A_{p},\dots,A_{p})}_{k}]\] and  $$\rho(\gamma_2)=
[(B_{p}^{s_{1}},\dots,B_{p}^{s_{k}})],\  \rho'(\gamma_2)=[(B_{p}^{-s_1},\dots,B_{p}^{-s_{m}})].$$ By similar reason as
in the proof of Example \ref{E:Ap-1}, $\rho$ and $\rho'$ are element-conjugate. By projecting to the first component
$\SU(p)/\langle\omega_{p}^{s_{1}}I\rangle$, it is shown in the proof of Example \ref{E:Ap-1} that $\rho$ and $\rho'$ are
not globally conjugate.
\end{proof}

\subsection{Proof of the main theorem}

Unacceptability of most groups in the following Prop. \ref{P:unacceptable1} is already shown by Larsen (\cite{Larsen1},
\cite{Larsen2}).

\begin{proposition}\label{P:unacceptable1}
Any group $G$ in the following list is unacceptable: \begin{itemize}
\item[(i)] $\SU(n)/\mu_{m}$ ($m|n$, $n\geq 3$, $m\geq 2$);
\item[(ii)] $\PSp(n)$ ($n\geq 3$);
\item[(iii)] $\SO(2n)$ ($n\geq 3$);
\item[(iv)] $\PSO(2n)$ ($n\geq 3$);
\item[(v)] $\Spin(n)$ ($n\geq 7$);
\item[(vi)] $\HSpin(4n)$ ($n\geq 2$);
\item[(vii)] $\F_4$, $\E_6$, $\E_7$, $\E_8$, $\E_{6}^{ad}$, $\E_{7}^{ad}$.
\end{itemize}
\end{proposition}

\begin{proof}
In item (i), $(n,m)$ falls into one of the following cases: (1)$m$ has an odd prime factor $p$;  (2)$n\geq 6$ and $2|m$;
(3)$n=4$ and $m\in\{2,4\}$. (1)when $m$ has an odd prime factor $p$, put $$A=\{[\diag\{\lambda_{1}I_{p},\dots,
\lambda_{n/p}I_{p}\}]\in G:|\lambda_{i}|=1\}.$$ Then, $Z_{G}(A)=A\cdot Z_{G}(A)_{\der}$ and $Z_{G}(A)_{\der}\cong
\SU(p)^{k}/\langle(\omega_{p}I,\cdots,\omega_{p}I)\rangle$. By Example \ref{E:Ap-2}, $Z_{G}(A)_{\der}$ is unacceptable.
By Lemmas \ref{L:centralizer2} and \ref{L:product}, so is $G$. (2)when $n\geq 6$ and $2|m$, put $$A=\{[\diag\{\lambda_{1}
I_{2},\dots,\lambda_{n/2}I_{2}\}]\in G:|\lambda_{i}|=1\}.$$ Then, $Z_{G}(A)=A\cdot Z_{G}(A)_{\der}$ and $Z_{G}(A)_{\der}
\cong\Sp(1)^{n/2}/\langle(-1,\dots,-1)\rangle$ with $\frac{n}{2}\geq 3$. By Example \ref{E:Sp3-2}, $\Sp(1)^{n/2}/\langle
(-1,\dots,-1)\rangle$ is unacceptable. By Lemmas \ref{L:centralizer2} and \ref{L:product}, so is $G$. (3)when $n=4$
and $m\in\{2,4\}$, by Example \ref{E:SU4} $G$ is unacceptable.

In item (ii), take $A\!=\!\{[\diag\{t_{1},\dots,t_{n}\}]:\!t_{j}\!=\!\pm{1}\}.$ Then, \[Z_{G}(A)\cong
\Sp(1)^{n}/\langle(-1,\dots,-1)\rangle.\] By Example \ref{E:Sp3-2}, $\Sp(1)^{n}/\langle(-1,\dots,-1)\rangle$ is
unacceptable. By Lemmas \ref{L:centralizer2} and \ref{L:product}, so is $G$.

In item (iii), take a maximal torus $A$ of $\SO(2n-6)\subset\SO(2n)=G$. Then, $Z_{\SO(2n)}(A)=A\times\SO(6)$.
By Example \ref{E:SU4}, $\SO(6)$ is unacceptable. By Lemmas \ref{L:centralizer2} and \ref{L:product}, so
is $G$.

In item (iv), when $n\geq 4$, take a maximal torus $A$ of $\SO(2n-6)\subset\PSO(2n)=G$. Then, $Z_{\PSO(2n)}(A)
=A\cdot\SO(6)$. By Example \ref{E:SU4}, $\SO(6)$ is un acceptable. By Lemmas \ref{L:centralizer2} and
\ref{L:product}, so is $G$. When $n=3$, $\PSO(6)\cong\PSU(4)$ is unacceptable by Example \ref{E:SU4}.

In item (v), take a maximal torus $A$ of $\Spin(n-7)\subset\Spin(n)=G$. Then, $Z_{\Spin(n)}(A)=
A\cdot Z_{\Spin(n)}(A)_{\der}$ and $Z_{\Spin(n)}(A)_{\der}\cong\Spin(7)$ (when $n$ is odd) or $\Spin(8)$ (when $n$
is even). By Example \ref{E:unacceptable-ChenevierGan} and Example \ref{E:Spin8}, both $\Spin(7)$ and $\Spin(8)$
are unacceptable. By Lemmas \ref{L:centralizer2} and \ref{L:product}, so is $G$.

In item (vi), when $n\geq 3$, take a maximal torus $A$ of $\Spin(4n-8)\subset\HSpin(4n)=G$. Then, $Z_{G}(A)=
A\cdot Z_{G}(A)_{\der}$ and $Z_{G}(A)_{\der}\cong\Spin(8)$ (due to the central element $c$ of $\Spin(4n)$ is not
contained in the subgroup $\Spin(8)$ of $\Spin(4n)$). By Example \ref{E:Spin8}, $\Spin(8)$ is unacceptable. By Lemmas
\ref{L:centralizer2} and \ref{L:product}, so is $G$. When $n=2$, $\HSpin(8)\cong\SO(8)$ is unacceptable.

In item (vii), each of $G=\E_6$, $\E_7$, $\E_8$, $\E_6^{ad}$, $\E_7^{\ad}$ possesses a unique Levi subgroup of type
$\bf{D}_{4}$ up to conjugacy. Thus, there exists a torus $A\subset G$ such that $Z_{G}(A)=A\cdot Z_{G}(A)_{\der}$
and $Z_{G}(A)_{\der}\cong\Spin(8)$, $\SO(8)$ or $\PSO(8)$ (actually more careful analysis shows that the derived
subgroup of this Levi subgroup is isomorphic to $\Spin(8)$). When $G=\F_4$, there exists a Klein four subgroup $A$
such that $Z_{G}(A)\cong\Spin(8)$ (\cite[Table 6]{Huang-Yu}). As each of $\Spin(8),\SO(8),\PSO(8)$ is unacceptable,
by Lemmas \ref{L:centralizer2} and \ref{L:product} so is $G$.
\end{proof}

We call a group in-decomposable if it is not the direct product of two nontrivial groups. Let $G$ be an
in-decomposable and non-simple connected compact semisimple Lie group. Then, $G$ is of the following form: \begin{equation}\label{Eq-standard}G=(G_1\times\cdots\times G_{s})/Z\end{equation} where $s\geq 2$, each $G_{i}$
($1\leq i\leq s$) is a connected compact simple Lie group, $Z\subset Z(G_1)\times\cdots\times Z(G_{s})$,
$Z\cap Z(G_{i})=1$ ($1\leq i\leq s$), and the image of projection of $Z$ to each $Z(G_{i})$ is non-trivial.

\begin{proposition}\label{P:unacceptable3}
Let $G$ be an in-decomposable and non-simple connected compact semisimple Lie group of the form in
\eqref{Eq-standard}. If $G$ is acceptable, then each $G_{i}\cong\Sp(1)$.
\end{proposition}

\begin{proof}
By Lemma \ref{L:product}, each $G_{i}$ is also acceptable. Since $G$ is in-decomposable, then each $Z(G_{i})\neq 1$.
Then by Prop. \ref{P:unacceptable1}, each $G_{i}\cong\SU(n)$ ($n\geq 3$) or $\Sp(n)$ ($n\geq 1$). Suppose that
not all $G_{i}$ ($1\leq i\leq s$) are isomorphic to $\Sp(1)$.

First, we show that $Z$ is a 2-group. Suppose it is not this case. Then, $Z$ contains an element with order an
odd prime $p$. Choose an order $p$ element $z=(z_{1},\dots,z_{s})\in Z$ with least number of nontrivial
components (among order $p$ elements in $Z$). Without loss of generality we assume that $z_{i}\neq 1$ if and only
$1\leq i\leq t$ where $1\leq t\leq s$. Then, the $p$-torsion subgroup of $Z\cap Z(G_{1}\times\cdots\times G_{t})$
is generated by $z$. Put $G_{i}=\SU(pm_{i})$ where $m_{i}\geq 1$ ($1\leq i\leq t$). Write $k:=m_{1}+\dots+m_{t}
\geq 1$. One can find a torus $A\subset G$ such that $Z_{G}(A)=A\cdot Z_{G}(A)_{\der}$ and $Z_{G}(A)_{\der}\cong
\SU(p)^{k}/\langle(\omega_{p}^{s_{1}}I,\cdots,\omega_{p}^{s_{k}}I)\rangle$, where $s_{1},\dots,s_{k}\in\{1,2,\dots,
p-1\}$. By Example \ref{E:Ap-2}, $Z_{G}(A)_{\der}$ is unacceptable. By Lemmas \ref{L:centralizer2} and
\ref{L:product}, so is $G$.

Second, we show that there exists an order 2 element $z=(z_{1},\dots,z_{s})\in Z$ and some $i$ such that
$z_{i}\neq 1$ and $G_{i}\not\cong\Sp(1)$. If $Z$ is not an elementary abelian 2-group, then it contains an element
$z'$ of order 4. Then, $z:=z'^{2}$ suffices the goal. If $Z$ is an elementary abelian 2-group. Choose $i$ such
that $G_{i}\not\cong\Sp(1)$. Since $G$ is in-decomposable, there exists $1\neq z=(z_{1},\dots,z_{s})\in Z$ such
that $z_{i}\neq 1$. Then, this $z$ suffices the goal.

Third, we deduce a contradiction. Without loss of generality we assume that $G_{1}\not\cong\Sp(1)$ and there
exists an order 2 element $z'=(z'_{1},\dots,z'_{s})\in Z$ such that $z'_{1}\neq 1$. Among such elements $z'$
choose an element $z=(z_1,\dots,z_{s})$ such that the number of non-trivial elements in $z_{2},\dots,z_{s}$
is the least. Without loss of generality we assume that $z_{i}\neq 1$ if and only $1\leq i\leq t$ where
$1\leq t\leq s$. Then, the 2-torsion subgroup of $Z\cap Z(G_{1}\times\cdots\times G_{t})$ is generated by $z$.
Put $G_{i}\cong\SU(2m_{i})$ or $\Sp(m_{i})$ ($1\leq i\leq t$) where $m_{i}\geq 1$ and $m_{1}\geq 2$. Since
$Z(G_{i})\cap Z=1$ for each $i$, we have $t\geq 2$. Then, $k:=m_{1}+\dots+m_{t}\geq 3$. One can find a closed
abelian subgroup $A\subset G$ such that $Z_{G}(A)=A\cdot Z_{G}(A)_{\der}$ and $Z_{G}(A)_{\der}\cong
\Sp(1)^{k}/\langle(-1,\dots,-1)\rangle$. By Example \ref{E:Sp3-2}, $\Sp(1)^{k}/\langle(-1,\dots,-1)\rangle$
is unacceptable. By Lemmas \ref{L:centralizer2} and \ref{L:product}, so is $G$. This contradicts to the
assumption of the proposition and finishes the proof.
\end{proof}

The following combinatorial Lemma \ref{L:F2} has nothing to do with acceptability. Its formulation is suggested by 
a referee. Let's view the set of subsets of $\{1,\dots,n\}$ as the $\mathbb{Z}/2\mathbb{Z}$-vector space
$V=(\mathbb{Z}/2\mathbb{Z})^{n}$ and denote by $e_{I}\in V$ for an element corresponding to a subset $I$ of
$\{1,\dots,n\}$. Then, $e_{I}+e_{J}=e_{(I-I\cap J)\sqcup(J-I\cap J)}$ for any two subsets $I$ and $J$ of
$\{1,\dots,n\}$. For a subset $I$ of $\{1,\dots,n\}$, put $V_{I}=\{e_{J}:J\subset I\}$. Then, $V_{I}$ is a
sub-vector space of $V$ with dimension $|I|$. For any sub-vector space $W$ of $V$, put $W_{I}=W\cap V_{I}$.

\begin{lemma}\label{L:F2}
Let $n\geq 2$ and $0\neq W$ be a sub-vector space of $V=(\mathbb{Z}/2\mathbb{Z})^{n}$. Assume that: \begin{enumerate}
\item[(i)]$W$ is in-decomposable: for any decomposition $\{1,\dots,n\}=I\sqcup J$ with $0<|I|<n$, $W\neq W_{I}\oplus W_{J}$.
\item[(ii)]for any $I$ such that $W_{I}=\langle e_{I}\rangle$, we have $|I|\leq 2$.
\item[(iii)]for any $I$ with $|I|=3$, $W_{I}$ is not the subspace of sum 0 elements in $V_{I}$.
\end{enumerate}
Then we have $n=2$.
\end{lemma}

\begin{proof}
As $W$ is in-decomposable and $V=V_{I}\oplus V_{J}$ whenever $\{1,\dots,n\}=I\sqcup J$, it follow that: for any $I$ with
$0<|I|<n$, \begin{enumerate}
\item[(1)]$W$ does not contain $V_{I}$.
\item[(2)]$W$ is not contained in $V_{I}$.
\item[(3)]when $|I|=1$, $e_{I}\not\in W$.
\end{enumerate}

Let $W'$ be the subspace of $W$ generated by $e_{I}\in W$ with $|I|=2$. We show that $W'=W$. Suppose $W'\neq W$. Choose
an element $e_{I}\in W-W'$ with $|I|$ minimal. By condition (3) in the first paragraph, $|I|>1$. By the definition of
$W'$, $|I|\neq 2$. Thus, $|I|\geq 3$. It is clear that $W_{I}=\langle e_{I}\rangle$. Then, this is in contradiction with
the assumption (ii).

Assume now that there exists $e_{I},e_{J}\in W$ ($I\neq J$) with $|I|=|J|=2$ and $I\cap J\neq\emptyset$. Put $K=I\cup J$.
Then, $W_{K}=\langle e_{I},e_{J}\rangle$ or $V_{K}$, which contradicts to the assumption (iii) and the condition (1)
respectively. As a consequence, those $I$ with $|I|=2$ and $e_{I}\in W$ form a partition of $\{1,2,\dots,n\}$. Indeed,
their union $U$ is $\{1,2,\dots,n\}$ by the condition (2) and $W=W'\subset V_{U}$. Then, $n$ is even and for such an
$I$, we have $W=W_{I}\oplus W_{J}$, where $J=\{1,2,\dots,n\}-I$. Thus, $|J|=0$ by the assumption (i). Hence, $n=2$.
\end{proof}

\begin{lemma}\label{L:Sp1}
If $G$ admits $\Sp(1)^{n}$ as a universal covering for some $n\geq 1$ and is acceptable, then it is isomorphic to a
direct product of $\Sp(1)$, $\Sp(1)/\langle-1\rangle$, $\Sp(1)^{2}/\langle(-1,-1)\rangle$.
\end{lemma}

\begin{proof}
We may assume that $G$ is in-decomposable. When $G$ is simple, then $G\cong\Sp(1)$ or $\Sp(1)/\langle-1\rangle$.
When $G$ is non-simple, it suffices to show that: if a group $G=\Sp(1)^{n}/Z$ ($n\geq 2$, $Z\subset Z(\Sp(1)^{n})$)
is in-decomposable, non-simple and acceptable, then $n=2$.

Identify $Z(\Sp(1)^{n})$ with the set of subsets of $\{1,\dots,n\}$, as well as with the
$\mathbb{Z}/2\mathbb{Z}$-vector space $V=(\mathbb{Z}/2\mathbb{Z})^{n}$. Let $W$ be the sub-vector space corresponding
to $Z$. Since $G$ is in-decomposable, then $W$ is in-decomposable, which is the assumption (i) in Lemma \ref{L:F2}.
By Example \ref{E:Sp3-2}, we get the assumption (ii) in Lemma \ref{L:F2}. By Example \ref{E:3A1-2}, we get the
assumption (iii) in Lemma \ref{L:F2}. By Lemma \ref{L:F2}, we get $n=2$.
\end{proof}

\smallskip

\begin{proof}[Proof of Theorem \ref{T1}]
The sufficiency follows form Prop. \ref{P:acceptable1} and Lemma \ref{L:Larsen}. For the necessarity, Prop.
\ref{P:unacceptable3} and Lemma \ref{L:Larsen} reduces it to the case that $G$ is a simple group or admits
$\Sp(1)^{n}$ as a universal covering. Then, Prop. \ref{P:unacceptable1} treats simple groups and Lemma \ref{L:Sp1}
treats quotients of $\Sp(1)^{n}$.
\end{proof}

\section{Non-connected compact Lie groups}\label{S:nonC}

\begin{proposition}\label{P:unacceptable2}
Any group $G$ in the following list is unacceptable: \begin{itemize}
\item[(i)] $\SU(2n)\rtimes\langle\tau\rangle$ ($n\geq 3$, $\tau^{2}=1$, $\Ad(\tau)X=\overline{X}$);
\item[(ii)] $\Pin(n)$ ($n\geq 7$),
\item[(iii)] $\PO(4n)$ ($n\geq 2$);
\item[(iv)] $\Spin(8)\rtimes\langle\tau\rangle$ ($\tau^3=1$, $\Spin(8)^{\tau}=\G_2$);
\item[(v)] $\PSO(8)\rtimes\langle\tau\rangle$ ($\tau^3=1$, $\PSO(8)^{\tau}=\G_2$);
\item[(vi)] $\E_{6}\rtimes\langle\tau\rangle$ ($\tau^2=1$, $\E_{6}^{\tau}=\F_4$);
\item[(vii)] $\E_{6}^{ad}\rtimes\langle\tau\rangle$ ($\tau^2=1$, $(\E_{6}^{ad})^{\tau}=\F_4$).
\end{itemize}
\end{proposition}

\begin{proof}
In item (i), we have $Z_{G}(\tau)=\SO(2n)\times\langle\tau\rangle.$ As $\SO(2n)$ ($n\geq 3$) is unacceptable, so
is $G$ by Lemmas \ref{L:centralizer2} and \ref{L:product}.

In item (ii), we have $Z_{\Pin(n)}(e_1)=\Spin(n-1)\cdot\langle e_1\rangle.$ When $n\geq 8$, $\Spin(n-1)$ is
unacceptable. Then, so is $G$ by Lemmas \ref{L:centralizer2} and \ref{L:product}. When $n=7$, $\Pin(7)=
\Spin(7)\cdot Z(\Spin(7))$. As $\Spin(7)$ is unacceptable, so is $\Pin(7)$ by Lemma \ref{L:product}.

In item (iii), choose $$A=\{[\left(\begin{array}{cc}aI_{2n}&bI_{2n}\\-bI_{2n}&aI_{2n}\\\end{array}\right)]:a,b\in
\mathbb{R},a^2+b^2=1\}.$$ Then, $Z_{G}(A)=A\cdot Z_{G}(A)_{\der}$ and $Z_{G}(A)_{\der}=\SU(2n)/\langle-I\rangle$.
By Prop. \ref{P:unacceptable1}(i), $\SU(2n)/\langle-I\rangle$ is unacceptable. Then by Lemma \ref{L:product},
so is $G$.

In items (iv)-(v), there exists an order 3 element $x\in G^{0}\tau$ such that $Z_{G}(x)=\PSU(3)\times
\langle x\rangle.$ As $\PSU(3)$ is unacceptable, so is $G$ by Lemmas \ref{L:centralizer2} and \ref{L:product}.

In items (vi)-(vii), $Z_{G}(\tau)=\F_{4}\times\langle\tau\rangle.$ As $\F_{4}$ is unacceptable, so is $G$ by
Lemmas \ref{L:centralizer2} and \ref{L:product}.
\end{proof}

In \cite{Larsen1} it is shown that the group $\SU(3)\!\rtimes\!\langle\tau\rangle$ where $\tau^2=1$ and $\Ad(\tau)(X)
=\overline{X}\ (\forall X\in\SU(3))$ is acceptable. Since $\Pin(5)=Z(\Pin(5))\cdot\Pin(5)_{\der}$ and $\Pin(5)_{\der}
\cong\Sp(2)$ is strongly acceptable, so is $\Pin(5)$ by Lemma \ref{L:Larsen}. The following Question \ref{Q:acceptable2}
asks the acceptability of some non-connected compact Lie groups.

\begin{question}\label{Q:acceptable2}
Are the following groups acceptable or unacceptable: \begin{itemize}
\item[(i)] $\Pin(2n)$ ($n=2$ or 3);
\item[(ii)]$\PO(4)$;
\item[(iii)] $\PO(4n+2)$ ($n\geq 1$).
\item[(iv)]$\U(n)\!\rtimes\!\langle\tau\rangle$ ($n\geq 2$, $\tau^{2}=1$, $\Ad(\tau)X=\overline{X}$);
\item[(v)]$\SU(\!2n\!+\!1)\!\rtimes\!\langle\tau\rangle$ ($n\geq 2$, $\tau^{2}=1$, $\Ad(\tau)X=\overline{X}$).
\end{itemize}
\end{question}


\section{$\SO_{4}(\mathbb{C})$-pseudocharacters and invariant functions}\label{S:pseudo}

Let $G$ be a connected complex linear reductive group. Write $\mathbb{C}[G^{n}]^{G}$ for the ring of regular
functions on $G^{n}=\underbrace{G\times\cdots\times G}_{n}$ invariant with respect to the diagonal conjugation
action of $G$ on $G^{n}$. A word $w$ of length $m$ and taking values in $\{1,2,\dots,n\}$ is a map
$w:\{1,\dots,m\}\rightarrow\{1,2,\dots,n\}$. Write $w_{i}=w(i)$ ($1\leq i\leq m$). For such a word and an
invariant function $f\in\mathbb{C}[G]^{G}$, define $\tilde{f}^{w}\in\mathbb{C}[G^{n}]^{G}$ by
\[\tilde{f}^{w}(g_{1},\dots,g_{n})=f(g_{w_{1}}\dots g_{w_{m}}),\] which is called a 1-argument invariant.
An invariant $f\in\mathbb{C}[G^{n}]^{G}$ is said to be generated by 1-argument invariants if it is a finite
linear combination of invariants of the form $\tilde{f}_{1}^{w_{1}}\cdots\tilde{f}_{k}^{w_{k}}$ where each $w_{j}$
($1\leq j\leq k$) is a word taking values in $\{1,2,\dots,n\}$, and $f_{1},\dots,f_{k}\in\mathbb{C}[G]^{G}$.



Let $\Gamma$ be a group. A complex coefficient $G$-pseudocharacter of $\Gamma$ is a collection of algebra homomorphisms $\Theta_{n}:\mathbb{C}[G^{n}]^{G}\rightarrow\Map(\Gamma^{n},\mathbb{C})$ ($n\geq 1$) satisfying the following two conditions
(cf. \cite[Definition 4.1]{Bockle-Harris-Khare-Thorne}): \begin{enumerate}
\item[(i)]for each $n,m\geq 1$, each map $\xi:\{1,\dots,m\}\rightarrow\{1,\dots,n\}$, each $f\in\mathbb{C}[G^{m}]^{G}$ and
any $\gamma_{1},\dots,\gamma_{n}\in\Gamma$, we have \[\Theta_{n}(f^{\xi})(\gamma_{1},\dots,\gamma_{n})=\Theta_{m}(f)
(\gamma_{\xi(1)},\dots,\gamma_{\xi(m)})\] where $f^{\xi}(g_{1},\dots,g_{n})=f(g_{\xi(1)},\dots,g_{\xi(m)})$.
\item[(ii)]for each $n\geq 1$, each $f\in\mathbb{C}[G^{n}]^{G}$ and any $\gamma_{1},\dots,\gamma_{n+1}\in\Gamma$,
\[\Theta_{n+1}(\hat{f})(\gamma_{1},\dots,\gamma_{n+1})=\Theta_{n}(f)(\gamma_{1},\dots,\gamma_{n-1},\gamma_{n}\gamma_{n+1}),
\] where $\hat{f}(g_{1},\dots,g_{n+1})=f(g_{1},\dots,g_{n-1},g_{n}g_{n+1})$.
\end{enumerate}

Suppose given a homomorphism $\rho:\Gamma\rightarrow G$. Then, the collection of maps $\Theta_{n}(f)=f(\rho(\gamma_1),\dots,
\rho(\gamma_{n}))$ is a $G$-pseudocharacter. Let $\tr\rho=(\Theta_{n})_{n\geq 1}$ denote the associated $G$-pseudocharacter
of a homomorphism $\rho:\Gamma\rightarrow G$.

\begin{theorem}[\cite{Lafforgue2}]\label{T:Lafforgue}
The assignment $\rho\mapsto\Theta=\tr\rho$ induces a bijection between the following two sets: \begin{enumerate}
\item[(i)]The set of $G$-conjugacy classes of $G$-completely reducible homomorphisms $\rho:\Gamma\rightarrow G$.
\item[(ii)]The set of $G$-pseudocharacters of $\Gamma$.
\end{enumerate}
\end{theorem}

Recall that, $\rho:\Gamma\rightarrow G$ is said to $G$-completely reducible if whenever $\Im\rho$ is contained in a
parabolic subgroup $P$ of $G$, then it is contained in a Levi subgroup of $P$. Using pseudocharacters (and other tools),
in \cite{Lafforgue2} Vincent Lafforgue showed one direction of Langlands correspondence over function fields: he
associates with each everywhere unramified, cuspidal automorphic representation its semisimplified Langlands parameter.
A consequence of Lafforgue's results is: for a connected complex linear reductive group $G$, if invariants in
$\mathbb{C}[G^{n}]^{G}$ are generated by 1-argument invariants for each $n\geq 1$, then $G$ is acceptable
(\cite{Lafforgue},\cite{Lafforgue2}). In \cite{Procesi2} Procesi showed that classical groups like $\GL_{n}(\mathbb{C})$,
$\Sp_{n}(\mathbb{C})$, $\O_{n}(\mathbb{C})$ satisfy this property. Procesi's result is used by Taylor to study
$\GL_{n}$-pseudocharacters. During a personal communication Xinwen Zhu asks the author if the converse statement
also holds (\cite{Zhu}). That is, if $G$ is acceptable, is $\mathbb{C}[G^{n}]^{G}$ generated by 1-argument invariants
for each $n\geq 1$? In the below we disprove this by showing that: for $G=\SO_{4}(\mathbb{C})$, $\mathbb{C}[G^{2}]^{G}$
is not generated by 1-argument invariants. Note that $\SO_{4}(\mathbb{C})$ is acceptable as its maximal compact subgroup
$\SO(4)$ is. Note that $\SO_{4}(\mathbb{C})\cong(\SL_{2}(\mathbb{C})\times\SL_{2}(\mathbb{C}))/\langle(-I,-I)\rangle$.
For simplicity, in the below let $\SL_{2}$ denote $\SL_{2}(\mathbb{C})$, and let $\SO_{4}$ denote $\SO_{4}(\mathbb{C})$.

\smallskip

The following Lemma \ref{L:SO4-1} is well-known.

\begin{lemma}\label{L:SO4-1}
Let $\SL_{2}$ act on $\SL_2\times\SL_{2}$ by diagonal conjugation. Write $(X_{1},X_{2})$ ($X_{1},X_{2}\in\SL_{2}$)
for a general element in $\SL_2\times\SL_{2}$. Then $\mathbb{C}[\SL_2\times\SL_{2}]^{\SL_2}$ is a polynomial ring with
generators $\tr X_{1},\tr X_{2},\tr X_{1}X_{2}.$
\end{lemma}

Write $G=\SL_2\times\SL_2$, $H=\SO_4$. Then, the map $(X,Y)\rightarrow Z=X\otimes Y$ gives a two-fold
covering $\pi:G\rightarrow H$.

\begin{lemma}\label{L:SO4-2}
Viewed as a subring of $\mathbb{C}[G]^{G}$, the ring $\mathbb{C}[H]^{H}$ is generated by $(\tr X)^{2}$, $(\tr Y)^{2}$,
$\tr X\tr Y.$
\end{lemma}

\begin{proof}
Let $\SL_{2}\times\SL_{2}$ act on itself by conjugation. Then \[\mathbb{C}[\SL_{2}\times\SL_{2}]^{\SL_{2}\times\SL_{2}}=
\mathbb{C}[\tr X,\tr Y].\] Write $z=(-I,-I)\in G$. Then, $H=G/\langle z\rangle$. Let $z$ act on $G$ by left translation.
Then the induced action of $z$ on $\mathbb{C}[G]^{G}$ is given by $$(z\tr X,\!z\tr Y)\!=
\!(-\tr X,\!-\tr Y).$$ Thus, \[\mathbb{C}[H]^{H}=(\mathbb{C}[G]^{G})^{z}=\mathbb{C}[(\tr X)^{2},(\tr Y)^{2},\tr X\tr Y].\]
\end{proof}

By elementary calculation, one can verify that Lemma \ref{L:SO4-2} is consistent with the invariant ring of $\SO_{4}$
given in \cite{Aslaksen-Tan-Zhu}.


Write $(X_1,Y_1,X_2,Y_2)$ for a general element of $G\times G=\SL_{2}^{4}$. By Lemma \ref{L:SO4-1}, the ring
\[R:=\mathbb{C}[G\times G]^{G}\] is a polynomial ring with generators $$\tr X_1,\tr X_2,\tr X_1X_2,\tr Y_1,\tr Y_2,
\tr Y_1Y_2.$$ Write $$x_1=\tr X_1,\quad x_2=\tr X_2,\quad x_3=\tr X_1X_2,$$ $$y_1=\tr Y_1,\quad y_2=
\tr Y_2,\quad y_3=\tr Y_1Y_2.$$

By a similar argument as in the proof of Lemma \ref{L:SO4-2}, one shows that the ring $R_{1}=
\mathbb{C}[H\times H]^{H}$ is generated by \begin{eqnarray*}&& x_1^{2},x_{2}^{2},x_{3}^{2},y_1^{2},y_{2}^{2},
y_{3}^{2},x_{1}y_{1},x_{2}y_{2},x_{3}y_{3},x_{1}x_{2}x_{3},y_{1}y_{2}y_{3},\\&& y_{1}x_{2}x_{3},x_{1}y_{2}x_{3},
x_{1}x_{2}y_{3},x_{1}y_{2}y_{3},y_{1}x_{2}y_{3},y_{1}y_{2}x_{3}.\end{eqnarray*}

Let $R_2$ be the sub-ring of $R_1$ generated by word maps of invariant functions on $G$. Write $w=(k_1,l_1,
\dots,k_{s},l_{s})$ for a word. By Lemma \ref{L:SO4-2}, $w$ gives invaraint functions $$f_{w}=(\tr X_{w})^{2},
g_{w}=(\tr Y_{w})^{2},h_{w}=\tr X_{w}\tr Y_{w}$$ where $$X_{w}=X_{1}^{k_1}X_{2}^{l_a}\cdots X_{1}^{k_{s}}X_{2}^{l_{s}},
\quad Y_{w}=Y_{1}^{k_1}Y_{2}^{l_a}\cdots Y_{1}^{k_{s}}Y_{2}^{l_{s}}.$$

For a word $w$, write $\epsilon=(-1)^{\sum_{1\leq i\leq s}k_{i}}$ and $\epsilon'=(-1)^{\sum_{1\leq i\leq s}l_{i}}$.
Then we have $$X_{w}(-X_1,X_2)=\epsilon X_{w}(X_1,X_2),\quad Y_{w}(-Y_1,Y_2)=\epsilon Y_{w}(Y_1,Y_2),$$ $$X_{w}(X_1,
-X_2)=\epsilon'X_{w}(X_1,X_2),\quad Y_{w}(Y_1,-Y_2)=\epsilon'Y_{w}(Y_1,Y_2).$$ We also have $$x_1(-X_1,X_2)=-x_1,\quad
x_2(-X_1,X_2)=x_2,\quad x_{3}(-X_1,X_2)=-x_3,$$ $$x_1(X_1,-X_2)=x_1,\quad x_2(X_1,-X_2)=-x_2,\quad x_{3}(X_1,-X_2)=
-x_3,$$ and similar relations for $y_1,y_2,y_3$. This shows that: $f_{w}$ is a sum of terms of the form
$$x_1^{a_1}x_2^{a_2}x_{3}^{a_3}$$ where $$a_{1}\equiv a_{2}\equiv a_{3} \pmod 2;$$ $g_{w}$ is a sum of terms of the
form $$y_1^{b_1}y_2^{b_2}y_{3}^{b_3}$$ where $$b_{1}\equiv b_{2}\equiv b_{3} \pmod 2;$$ $h_{w}$ is a sum of terms of
the form $$x_1^{a_1}x_2^{a_2}x_{3}^{a_3}y_1^{b_1}y_2^{b_2}y_{3}^{b_3}+x_1^{b_1}x_2^{b_2}x_{3}^{b_3}y_1^{a_1}y_2^{a_2}
y_{3}^{a_3}$$ where $$a_1+b_{1}\equiv a_2+b_{2}\equiv a_3+b_{3} \pmod 2.$$ By this, $f_{w}$ is in the ring generated by $x_1^{2},x_2^{2},x_3^{2},x_1x_2x_3$, and $g_{w}$ is in the ring generated by $y_1^{2},y_2^{2},y_3^{2},y_1y_2y_3$.

\begin{proposition}\label{P:SO4-1}
We have $R_1\neq R_2$.
\end{proposition}

\begin{proof}
Make the polynomial algebra $R$ a graded algebra by letting each of $x_1,x_2,x_3,y_1,y_2,y_3$ having degree 1. Write
$(R_{i})_{k}$ for the degree $k$ part of $R_{i}$. By the above description for $R_{1}$ and $R_{2}$, one has
$\dim (R_1)_{3}=8$ and $\dim (R_2)_{3}\leq 5$. Thus, $R_{2}$ is a proper subalgebra of $R_1$.
\end{proof}

Prop. \ref{P:SO4-1} implies the following.

\begin{corollary}
There are invariants in $\mathbb{C}[\SO_{4}\times\SO_{4}]^{\SO_4}$ which are not generated by 1-argument invariants.
\end{corollary}

In the $\SO_{4}$ case as above, $R_{2}$ contains the ring generated by \begin{eqnarray*}&& x_1^{2},
x_{2}^{2},x_{3}^{2},y_1^{2},y_{2}^{2},y_{3}^{2},x_{1}y_{1},x_{2}y_{2},x_{3}y_{3},x_{1}x_{2}x_{3},y_{1}y_{2}y_{3},\\&&
y_{1}x_{2}x_{3}+x_{1}y_{2}y_{3},x_{1}y_{2}x_{3}+y_{1}x_{2}y_{3},x_{1}x_{2}y_{3}+y_{1}y_{2}x_{3}.\end{eqnarray*} From
this, we see that the fractional fields of $R_1$ and $R_{2}$ are equal. Since $R_{1}$ is defined by invariants of a
group action on polynomial ring, it follows that it is a normal ring. Thus, $R_{1}$ is the integral closure of $R_{2}$.
This might be a general fact. It is interesting to determine $R_2$ fully. Particularly, do we have $(R_1)_{k}=(R_{2})_{k}$
whenever $k\neq 3$? More complete knowledge concerning invariants in $\mathbb{C}[\SO_{4}^{n}]^{\SO_{4}}$ is contained in
\cite{Aslaksen-Tan-Zhu}.

\begin{question}\label{Q:G2}
Let $G=\G_{2}(\mathbb{C})$ be a connected complex simple Lie group of type $\mathbf{G}_2$. Are invariants in $\mathbb{C}[G^{n}]^{G}$
generated by 1-argument invariants for any $n>1$.
\end{question}

The answer to Question \ref{Q:G2} is perhaps negative. Note that for $G=\G_{2}(\mathbb{C})$, invariants in
$\mathbb{C}[G^{n}]^{G}$ are determined by Schwarz (\cite{Schwarz}).

\end{document}